\documentclass[12pt,a4paper,reqno]{amsart}
\usepackage[english]{babel}
\usepackage[applemac]{inputenc}
\usepackage[T1]{fontenc}
\usepackage{palatino}
\usepackage{amsmath}
\usepackage{amssymb}
\usepackage{amsthm}
\usepackage{amsfonts}
\usepackage{graphicx}
\usepackage[colorlinks = true, citecolor = black]{hyperref}
\author{Tuomas Orponen}\thanks{The research was supported by the Finnish National Doctoral Programme in Mathematics and its Applications}
\title{On the Packing Measure of Self-Similar Sets}
\address{Department of Mathematics and Statistics, University of Helsinki, P.O.B. 68, FI-00014 Helsinki, Finland}
\email{tuomas.orponen@helsinki.fi} 

\newcommand{\R}{\mathbb{R}}
\newcommand{\N}{\mathbb{N}}
\newcommand{\Q}{\mathbb{Q}}

\newcommand{\tn}{\mathbf{P}}

\newcommand{\spt}{\operatorname{spt}}
\newcommand{\Hd}{\dim_{\mathrm{H}}}
\newcommand{\calP}{\mathcal{P}}

\numberwithin{equation}{section}

\theoremstyle{plain}
\newtheorem{thm}[equation]{Theorem}
\newtheorem{lemma}[equation]{Lemma}

\theoremstyle{definition}

\theoremstyle{remark}

\begin{document}

\maketitle

\begin{abstract} Building on a recent result of M. Hochman \cite{Ho}, we give an example of a self-similar set $K \subset \R$ such that $\Hd K = s \in (0,1)$ and $\calP^{s}(K) = 0$. This answers a question of Y. Peres and B. Solomyak.
\end{abstract}

\section{Introduction}

In 1998, B. Solomyak \cite{So} showed that there exist self-similar sets $K \subset \R$ of Hausdorff dimension $\Hd K = s \in (0,1)$, which have zero $s$-dimensional Hausdorff measure. In 2000, Y. Peres, K. Simon and Solomyak \cite{PSS} proved that such examples are not even all that rare: for certain natural families of self-similar sets, a large proportion of all sets in the family have this property. 

For packing measure, things are different. In fact, the same theorem in \cite{PSS} contained the further statement that almost all sets in these families have positive and finite $s$-dimensional packing measure. Until a recent breakthrough article of M. Hochman \cite{Ho}, it appeared hard to determine whether this statement could be further strengthened as follows: all self-similar sets $K$ with $\Hd K = s \in (0,1)$ have positive $s$-dimensional packing measure, denoted by $\calP^{s}$. The question is explicitly stated in \cite[Question 2.3]{PS}. In this note, building on Hochman's paper, we answer the question in the negative by exhibiting an explicit counterexample. In fact, we find a self-similar set $K \subset \R$ with \emph{similarity dimension} $s := \log 3/\log 4$ such that $K$ has 'no total overlaps' and $\calP^{s}(K) = 0$. Then, we employ Hochman's result to conclude that $K$ has Hausdorff dimension $s$. More precisely, we use the following theorem, the proof of which is the same as \cite[Theorem 1.6]{Ho}, apart from changing some numerical values:
\begin{thm}\label{hochman} Let $K_{u} \subset \R$ be the self-similar set generated by the three similitudes
\begin{displaymath} \psi_{0}(x) := \frac{x}{4}, \quad \psi_{1}(x) := \frac{x + 1}{4}, \quad \psi_{u}(x) := \frac{x + u}{4}, \end{displaymath}
where $u \in [0,1]$. Then $\Hd K_{u} = \log 3/\log 4$ for every $u \in \R \setminus \Q$.
\end{thm}

\section{Acknowledgements}

I am grateful to the referees for their constructive suggestions.

\section{Self-similar sets and packing measures} 

A set $K \subset \R$ is \emph{self-similar}, if $K$ is compact and satisfies the equation
\begin{equation}\label{selfSimilar} K = \bigcup_{j = 1}^{q} \psi_{j}(K), \end{equation}
where the mappings $\psi_{j} \colon \R \to \R$, $1 \leq j \leq q$, are contracting similitudes. This means that $\psi_{j}$ has the form $\psi_{j}(x) = r_{j}x + a_{j}$ for some \emph{contraction ratio} $r_{j} \in (-1,1)$ and \emph{translation vector} $a_{j} \in \R$. For a given set of contracting similitudes, there is one and only one non-empty compact set $K$ satisfying \eqref{selfSimilar}; this foundational result is due to J. Hutchinson \cite{Hu}. The \emph{similarity dimension} of a self-similar set $K$, as in \eqref{selfSimilar}, is the unique number $s \geq 0$ satisfying
\begin{displaymath} \sum_{j = 1}^{q} r_{j}^{s} = 1. \end{displaymath}
The Hausdorff dimension of $K$, denoted by $\Hd K$, is always bounded from above by the similarity dimension, see \cite[Theorem 9.3]{Fa}. In general, however, it is a hard problem to determine when the two dimensions coincide. A recent breakthrough in this respect is Hochman's paper \cite{Ho}, containing many satisfactory answers, including but not limited to Theorem \ref{hochman}. 

On a self-similar set $K$, as in \eqref{selfSimilar}, there is supported a \emph{natural self-similar measure} $\mu$, satisfying the equation
\begin{displaymath} \mu = \sum_{j = 1}^{q} r_{j} \cdot \psi_{j\sharp}\mu. \end{displaymath}
Here $\psi_{j\sharp}\mu$ is the \emph{push-forward} of $\mu$ under $\psi_{j}$, defined by $\psi_{j\sharp}\mu(B) = \mu(\psi_{j}^{-1}(B))$ for $B \subset \R$. We will be concerned with the question, when (or when not) $\mu$ is absolutely continuous with respect to the \emph{$s$-dimensional packing measure on $\R$}, denoted by $\calP^{s}$. The definition of $\calP^{s}$ is not directly used in the paper, but we include it here for completeness. First, one defines the \emph{$s$-dimensional packing pre-measure} $P^{s}$ by
\begin{displaymath} P^{s}(K) = \lim_{\delta \to 0} \sup \left\{ \sum_{i = 1}^{\infty} d(B_{i})^{s} : B_{i} \in \mathcal{D}_{\delta}(K) \text{ are disjoint} \right\}, \end{displaymath}
where $\mathcal{D}_{\delta}(K)$ is the collection of balls centred at $K$, with $d(B_{i}) \leq \delta$. The pre-measure $P^{s}$ is not countably additive, unfortunately, and this is the reason for defining
\begin{displaymath} \calP^{s}(K) := \inf \left\{ \sum_{i = 1}^{\infty} P^{s}(B_{i}) : K \subset \bigcup_{i = 1}^{\infty} B_{i} \right\}. \end{displaymath}
For more information on packing measures, see the book chapters \cite[\S 5.10]{Ma} and \cite[\S 3.4]{Fa}.

We will occasionally use the notation $A \lesssim B$ to mean that $A \leq CB$ for some absolute consonant $C \geq 1$. The two-sided inequality $A \lesssim B \lesssim A$ is abbreviated to $A \asymp B$.

\section{The Construction of $K$ and some reductions}

The self-similar set $K$, which answers \cite[Question 2.3]{PS} negatively, is generated by the three similitudes $\psi_{0}$, $\psi_{1}$ and $\psi_{u}$ as introduced in Theorem \ref{hochman}. The parameter $u \in [0,1]$ is chosen as follows: pick natural numbers $\lambda_{j} \in \{3^{3^{j}},3^{3^{j}} + 1\}$ in such a manner that
\begin{equation}\label{U} u := \sum_{j = 1}^{\infty} 4^{-\lambda_{j}} \in \R \setminus \Q. \end{equation}  
This is certainly possible, since there are uncountably many admissible sequences $(\lambda_{1},\lambda_{2},\ldots)$, and no two sequences produce the same number $u$. In fact, as pointed out by one of the referees, the choice $\lambda_{j} = 3^{3^{j}}$ is admissible, since the base-$4$ expansion of $u$ so obtained is not eventually periodic. Theorem \ref{hochman} now implies that $\Hd K = \log 3/\log 4 =: s$, so it remains to prove that $\calP^{s}(K) = 0$.

According to \cite[Corollary 2.2]{PSS}, if $\mathcal{P}^{s}(K) > 0$, then the natural self-similar measure $\mu$ supported on $K$ coincides with a normalised version of the restriction of $\mathcal{P}^{s}$ to $K$. In particular, this means that $\mu \ll \mathcal{P}^{s}$. Using \cite[Chapter 6, Exercise 5]{Ma}, we can then infer that
\begin{displaymath} \Theta^{s}_{\ast}(\mu,x) := \liminf_{r \to 0} \frac{\mu(B(x,r))}{(2r)^{s}} < \infty \end{displaymath} 
for $\mu$ almost every $x \in \R$. Thus, in order to show that $\calP^{s}(K) = 0$, we need to verify the following theorem.
\begin{thm}\label{mainC} Let $\mu$ be the natural self-similar measure on $\R$ associated with the system $\{\psi_{0},\psi_{1},\psi_{u}\}$. Then $\Theta_{\ast}^{s}(\mu,x) = \infty$ at $\mu$ almost every point $x \in \R$.
\end{thm}
The condition $\Theta^{s}_{\ast}(\mu,x) = \infty$ has a natural geometric interpretation, which will be formulated in the next lemma. First we need to introduce some notation. Let $\mathcal{I}_{0} = \{[0,1]\}$, and, for $n \geq 1$, define the collection of intervals
\begin{displaymath} \mathcal{I}_{n} := \{\psi_{i_{1}} \circ \ldots \circ \psi_{i_{n}}([0,1]) : (i_{1},\ldots,i_{n}) \in \{0,1,u\}^{n}\} \end{displaymath}
for $n \geq 1$. Then, it is easy to verify that the natural self-similar measure $\mu$ on $K$ has the property that
\begin{displaymath} \mu(J) \geq \frac{\sharp \{I \in \mathcal{I}_{n} : I \subset J\}}{3^{n}} \end{displaymath}
for any interval $J \subset \R$ and any $n \in \N$.
\begin{lemma} Let $x \in \R$. Assume that there exists $C \geq 1$ such that 
\begin{equation}\label{mainEq} \liminf_{n \to \infty} \sharp \{(i_{1},\ldots,i_{n}) \in \{0,1,u\}^{n} : \psi_{i_{1}} \circ \ldots \circ \psi_{i_{n}}(0) \in B(x,C4^{-n})\} = \infty. \end{equation} 
Then $\Theta_{\ast}^{s}(\mu,x) = \infty$.
\end{lemma}

\begin{proof} Given $M > 0$, we find $n_{M} \in \N$ such that at least $M$ distinct points of the form $\psi_{i_{1}} \circ \ldots \circ \psi_{i_{n}}(0)$ are contained in $B(x,C4^{-n})$ for $n \geq n_{M}$. The corresponding intervals in $\mathcal{I}_{n}$ have length $4^{-n}$, and are thus contained in $B(x,(C + 1)4^{-n})$. This shows that
\begin{displaymath} \frac{\mu(B(x,(C + 1)4^{-n}))}{(2(C + 1)4^{-n})^{s}} \geq [2(C + 1)]^{-s} \cdot \frac{M}{3^{n}} \cdot 4^{ns} = \frac{M}{[2(C + 1)]^{s}}, \qquad n \geq n_{M}, \end{displaymath} 
which finishes the proof.
\end{proof}

\section{Proof of $\mathcal{P}^{s}(K) = 0$}

The goal of this section is to demonstrate that the condition in \eqref{mainEq} is met at $\mu$ almost every point $x \in \R$. Let us begin by introducing some further notation and terminology. Write $\Omega := \{0,1,u\}^{\N}$, and let $\pi \colon \Omega \to \spt \mu = K$ be the projection
\begin{displaymath} \pi(\omega_{1},\omega_{2},\ldots) := \lim_{n \to \infty} \psi_{\omega_{1}} \circ \ldots \circ \psi_{\omega_{n}}(0) = \sum_{n = 1}^{\infty} \omega_{n} \cdot 4^{-n}. \end{displaymath}
Then $\mu = \pi_{\sharp}\tn$, where $\tn$ is the equal-weights product measure on $\Omega$. Let $\omega = (\omega_{1},\omega_{2},\ldots) \in \Omega$, and let $i,j \in \N$ be indices with $i \leq j$. We say that $(\omega,j)$ \emph{is influenced by $(\omega,i)$}, if there exists $k \in \N$ such that $i + \lambda_{k} < j \leq i + \lambda_{k + 1}$, and
\begin{displaymath} (\omega_{i},\omega_{i + \lambda_{1}},\ldots,\omega_{i + \lambda_{k}}) = (u,\underbrace{0,\ldots,0}_{k \text{ zeroes}}). \end{displaymath} 
Then, define
\begin{displaymath} S(\omega,j) := \sharp \{1 \leq i \leq j : (\omega,j) \text{ is influenced by } (\omega,i)\}. \end{displaymath} 
The point of this definition is, as we shall see later, that \eqref{mainEq} holds for all points $x = \pi(\omega) \in K$ such that 
\begin{equation}\label{mainEq2} \liminf_{j \to \infty} S(\omega,j) = \infty. \end{equation}
This in mind, we need to establish the following.
\begin{lemma}\label{mainC2} The equation \eqref{mainEq2} is valid $\tn$ almost surely.
\end{lemma}   

\begin{proof} We will be done as soon as we show that
\begin{displaymath} \tn[\{\omega : \liminf_{j \to \infty} S(\omega,j) \leq M\}] = 0 \end{displaymath} 
for any given $M \in \N$. We note that 
\begin{displaymath} \{\omega : \liminf_{j \to \infty} S(\omega,j) \leq M\} \subset \limsup_{j \to \infty} B_{j,M}, \end{displaymath} 
where $B_{j,M}$ is the set
\begin{displaymath} B_{j,M} = \{\omega : S(\omega,j) \leq M\}. \end{displaymath} 
Using the Borel-Cantelli lemma, we get $\tn[\limsup B_{j,M}] = 0$, if we manage to prove that
\begin{equation}\label{aux1} \sum_{j = 1}^{\infty} \tn[B_{j,M}] < \infty. \end{equation}
Thus, \eqref{aux1} will imply Lemma \ref{mainC2}. 

We are aiming at an upper bound for $\tn[B_{j,M}]$. Fix $j \geq \lambda_{1}$ and choose $k = k(j) \in \N \cup \{0\}$ so that $\lambda_{k + 1} \leq j < \lambda_{k + 2}$. Then divide the natural numbers between $j - \lambda_{k + 1}$ and $j - 1$ into consecutive blocks $I_{1},\ldots,I_{N}$ of length $|I_{j}| \in [\lambda_{k} + 1, 2\lambda_{k}]$. Let $i_{1},\ldots,i_{N}$ be the smallest numbers in these blocks. Then 
\begin{displaymath} N \asymp \lambda_{k + 1}/\lambda_{k} \asymp 3^{3^{k + 1}-3^{k}} \geq 3^{3^{k}}. \end{displaymath}
Now, the blocks $I_{n}$ are disjoint, so the random variable $X_{j} \colon \Omega \to \N$ defined by
\begin{displaymath} X_{j}(\omega) = \sharp\{1 \leq n \leq N : (\omega_{i_{n}},\omega_{i_{n} + \lambda_{1}},\ldots,\omega_{i_{n} + \lambda_{k}}) = (u,0,\ldots,0)\} \end{displaymath}
has distribution $X_{j} \sim \operatorname{Bin}(N,p_{j})$, where the 'success probability' $p_{j}$ equals
\begin{displaymath} p_{j} := \tn[\{\omega : (\omega_{i_{n}},\omega_{i_{n} + \lambda_{1}},\ldots,\omega_{i_{n} + \lambda_{k}}) = (u,0,\ldots,0)\}] = 3^{-k - 1}. \end{displaymath} 

\noindent We claim that $B_{j,M} \subset \{X_{j} \leq M\}$. Indeed, suppose that $\omega \in B_{j,M}$. Then, in particular, there are at most $M$ among the numbers $i_{n}$, $1 \leq n \leq N$, such that $(\omega,j)$ is influenced by $(\omega,i_{n})$. For the rest of the numbers $i_{n}$ either
\begin{equation}\label{aux2} (\omega_{i_{n}},\omega_{i_{n} + \lambda_{1}},\ldots,\omega_{i_{n} + \lambda_{k}}) \neq (u,0,\ldots,0) \end{equation}

\noindent or $i_{n} + \lambda_{k + 1} < j$, or $j \leq i_{n} + \lambda_{k}$, by definition of the notion of 'influence'. The latter two possibilities are absurd, since $i_{n} \geq j - \lambda_{k + 1}$ and $i_{n} + \lambda_{k} \in I_{n} \subset \{1,\ldots,j - 1\}$. Thus, \eqref{aux2} must hold for all but at most $M$ numbers $i_{n}$, which means that precisely that $X_{j}(\omega) \leq M$. 

The probability $\tn[\{X_{j} \leq M\}]$ can be estimated by a standard tail bound for the binomial distribution (or see Hoeffding's inequality \cite{Hoe}):
\begin{displaymath} \tn[\{X_{j} \leq M\}] \leq \exp\left(-2\frac{(N3^{-k - 1} - M)^{2}}{N}\right) \leq C_{M}\exp(-3^{3^{k} - 2k}), \end{displaymath} 
so that finally
\begin{align*} \sum_{j = 1}^{\infty} \tn[B_{j,M}] & \leq \sum_{j = 1}^{\lambda_{1} - 1} \tn[B_{j,M}] + \sum_{k = 0}^{\infty} \sum_{\lambda_{k + 1} \leq j < \lambda_{k + 2}} \tn[\{X_{j} \leq M\}]\\
& \leq \lambda_{1} + C_{M}\sum_{k = 0}^{\infty} 3^{3^{k + 2}}\exp(-3^{3^{k} - 2k}) < \infty.  \end{align*} 
This proves Lemma \ref{mainC2}. \end{proof}

Finally, it is time to check that our definition of 'influence' is useful:
\begin{lemma}\label{claimF} Assume that $\omega \in \Omega$ satisfies \eqref{mainEq2}. Then $\pi(\omega) \in K$ satisfies \eqref{mainEq}.
\end{lemma} 

\begin{proof} Fix $\omega \in \Omega, j > 1$, and write $x = \pi(\omega)$. Our task is to find many sequences $(i_{1},\ldots,i_{j}) \in \{0,1,u\}^{j}$ such that 
\begin{displaymath} \psi_{i_{1}} \circ \ldots \circ \psi_{i_{j}}(0) = \sum_{n = 1}^{j} i_{n} \cdot 4^{-n} \in B(x,C4^{-j}), \end{displaymath}

\noindent where $C \geq 1$ is some absolute constant. One such sequence is always obtained by taking $(i_{1},\ldots,i_{j}) = (\omega_{1},\ldots,\omega_{j})$, since
\begin{equation}\label{form1} \left|\sum_{n = 1}^{j} \omega_{n} \cdot 4^{-n} - x \right| = \sum_{n = j + 1}^{\infty} \omega_{n} \cdot 4^{-n} \asymp 4^{-j}. \end{equation} 
Here is how we find other sequences. Suppose that $(\omega,j)$ is influenced by $(\omega,i)$ for some $1 \leq i \leq j$. Once again, this means that $i + \lambda_{k} < j \leq i + \lambda_{k + 1}$ and $(\omega_{i},\omega_{i + \lambda_{1}},\ldots,\omega_{i + \lambda_{k}}) = (u,0,\ldots,0)$. Consider the modified sequence $\tilde{\omega}$, which is otherwise identical with $\omega$, except that the symbol $u$ at index $i$ is replaced by $0$, and the zeroes at the indices $i + \lambda_{1},\ldots,i + \lambda_{k}$ are replaced by $1$. Then, using the definition of $u$, we have
\begin{align*} \Big|\sum_{n = 1}^{j} \omega_{n} \cdot 4^{-n} & - \sum_{n = 1}^{j} \tilde{\omega}_{n} \cdot 4^{-n} \Big|\\
& = |(u - 0) \cdot 4^{-i} + (0 - 1) \cdot 4^{-(i + \lambda_{1})} + \ldots + (0 - 1) \cdot 4^{-(i + \lambda_{k})} |\\
& = \left|\sum_{n = 1}^{\infty} 4^{-i - \lambda_{n}} - \sum_{n = 1}^{k} 4^{-i - \lambda_{k}} \right| = \sum_{n = k + 1}^{\infty} 4^{-i - \lambda_{n}} \lesssim 4^{-j}, \end{align*} 
since $i + \lambda_{k + 1} \geq j$. It follows from this and \eqref{form1} that $\psi_{\tilde{\omega}_{1}} \circ \ldots \circ \psi_{\tilde{\omega}_{j}}(0) \in B(x,C4^{-j})$ for some absolute constant $C \geq 1$. Thus, for each pair $(\omega,i)$ influencing the pair $(\omega,j)$, the construction just described produces a sequence $(i_{1},\ldots,i_{j})$ with $\psi_{i_{1}} \circ \ldots \circ \psi_{i_{j}}(0) \in B(x,C4^{-j})$. 

Moreover, no sequence $(i_{1},\ldots,i_{j})$ is obtained twice in this manner. For if $(\omega,j)$ is influenced by both $(\omega,i)$ and $(\omega,i')$ with $i < i'$, say, then 
\begin{itemize}
\item $\omega_{i} = u$, by definition of $(\omega,i)$ influencing $(\omega,j)$, and
\item both $i$ and $i'$ give rise to modified sequences $\tilde{\omega}$ and $\tilde{\omega}'$, as above. 
\end{itemize}
Recalling how these sequences were constructed, we see that $\tilde{\omega}_{i} = 0$. On the other hand, $\tilde{\omega}'$ coincides with $\omega$ for all indices smaller than $i'$, so in particular $\tilde{\omega}'_{i} = \omega_{i} = u$. This means that $\tilde{\omega} \neq \tilde{\omega}'$ and completes the proof of the claim. \end{proof}

To conclude the proof of $\mathcal{P}^{s}(K) = 0$, we note that, by Lemma \ref{claimF}, the set $G = \{x : \eqref{mainEq} \text{ holds at } x\}$ contains the $\pi$-images of all those sequences $\omega \in \Omega$ where \eqref{mainEq2} holds. The set consisting of such sequences has full $\tn$-measure according to Lemma \ref{mainC2}. Hence, the equation $\mu = \pi_{\sharp}\tn$ implies full $\mu$-measure for $G$.



\end{document}